\documentclass[12pt, reqno,draft]{amsart}

\usepackage{amsmath}
\usepackage{amsthm}
\usepackage{amssymb}
\usepackage{amsrefs}
\usepackage{mathrsfs}
\usepackage[usenames]{color}

 \newtheorem{thm}{}[section]
 \newtheorem{theorem}[thm]{Theorem}
 \newtheorem{corollary}[thm]{Corollary}
 \newtheorem{lemma}[thm]{Lemma}
 \newtheorem{proposition}[thm]{Proposition}
  
  \theoremstyle{definition}
 
 \theoremstyle{remark}
 \newtheorem{remark}[thm]{Remark}
 \newtheorem{problem}[thm]{Problem}

 \numberwithin{equation}{section}
 \allowdisplaybreaks

\newcommand{\norm}[1]{\lVert#1\rVert}

\newcommand{\NN}{\ensuremath{\mathbb{N}}}

\newcommand{\xx}{\ensuremath{\mathbf{x}}}

\newcommand{\ee}{\ensuremath{\mathbf{e}}}
\newcommand{\ww}{\ensuremath{\mathbf{w}}}
\newcommand{\vv}{\ensuremath{\mathbf{v}}}

\newcommand{\WW}{\ensuremath{\mathcal{W}}}
\newcommand{\OO}{\ensuremath{\mathcal{O}}}
\newcommand{\EE}{\mathcal{E}}
\newcommand{\BB}{\mathcal{B}}
\newcommand{\HH}{\mathcal{H}}

\newcommand{\FF}{\mathbb{F}}
\newcommand{\FFF}{\mathcal{F}}

\let\<=\langle
\let\>=\rangle

\newcommand{\ZZ}{\ensuremath{\mathbb{Z}}}

\newcommand{\supp}{\operatorname{supp}}
\newcommand{\range}{\operatorname{R}}
\newcommand{\dom}{\operatorname{D}}

\newcommand{\normiii}[1]{{\left\vert\mkern-1.8mu\left\vert\mkern-1.8mu\left\vert #1 
 \right\vert\mkern-1.8mu\right\vert\mkern-1.8mu\right\vert}}

\begin{document}

\title{$1$-Greedy renormings of Garling sequence  spaces}

\author[F. Albiac]{Fernando Albiac}
\address{Mathematics Department\\ 
Universidad P\'ublica de Navarra\\
Campus de Arrosad\'{i}a\\
Pamplona\\ 
31006 Spain}
\email{fernando.albiac@unavarra.es}

\author[J. L. Ansorena]{Jos\'e L. Ansorena}
\address{Department of Mathematics and Computer Sciences\\
Universidad de La Rioja\\ 
Logro\~no\\
26004 Spain}
\email{joseluis.ansorena@unirioja.es}

\author[B.  Wallis]{Ben Wallis}
\address{University of North Illinois\\ 
USA}



\begin{abstract} 
 We show that  all 
 Garling sequence  spaces admit a renorming with respect to which their standard unit vector basis is $1$-greedy. We also discuss some additional properties of these Banach spaces related to uniform convexity and superreflexivity. In particular, our approach to the study of the superreflexivity  of Garling sequence space provides an example of how  essentially non-linear tools  from greedy approximation can be used to shed light into the linear structure of the spaces.
\end{abstract} 

\thanks{Research partially supported by the Spanish Research Grant \textit{An\'alisis Vectorial, Multilineal y Aplicaciones}, reference number MTM2014-53009-P. F. Albiac also acknowledges the support of Spanish Research Grant \textit{Operators, lattices, and structure of Banach spaces},  with reference MTM2016-76808-P}

\maketitle

\section{Introduction and background}

\noindent A semi-normalized basis $(\xx_{n})_{n=1}^{\infty}$ of a Banach space $(X, \Vert\cdot\Vert)$ is  said to be $C$-\textit{greedy under renorming} ($C$-GUR, for short) if there is an equivalent norm $\normiii{\cdot}$ on $X$ (i.e., a renorming of $X$) with respect to which $(\xx_{n})_{n=1}^{\infty}$ is $C$-greedy, i.e., 
\[
\normiii { f-\sum_{n\in A}  a_n \xx_n}
\le C \normiii{ f-g}
\]
 for any $f=\sum_{n=1}^\infty a_n\xx_n\in X$, any $A\subseteq \NN$ finite such that  $|a_n|\ge  |a_k|$ whenever $n\in A$ and $k\in\NN\setminus A$, and any $g\in X$ with $|\supp(g)|\le|A|$.

A problem that goes back to \cite{AW2006} is to determine if  a given  (greedy) basis  is $1$-GUR.
For symmetric bases the answer to this problem  is positive and quite simple because $C$-symmetric bases are $C$-greedy and every symmetric basis becomes $1$-symmetric under a suitable renorming; thus  any symmetric basis is $1$-GUR.

For subsymmetric bases  the situation is  different. Taking into account the relation between the constants involved (see e.g. \cite{AlbiacKalton2016}*{Chapter 10}) one immediately sees that $1$-subsymmetric bases are always $2$-greedy.
 Hence, since any subsymmetric basis becomes $1$-subsymmetric under a suitable renorming, we have that any subsymmetric basis is $2$-GUR.  
 
 Let us now put our problem into context by summarizing its backgroung.  Albiac and Wojtaszczyk  exhibited in  \cite{AW2006}  an example of a $1$-subsymmetric basis that  is not $1$-greedy.    Later on,  Dilworth et al.\  constructed in  \cite{DOSZ2011} an example of a  subsymmetric basis which,  in spite of not being symmetric, was $1$-greedy.
 Therefore a natural question in the theory is to determine if a particular subsymmetric (and non-symmetric) basis is $1$-GUR.


Recently,  the authors have investigated in  \cite{Wallis2017} the geometric properties of  a class of Banach  spaces, called Garling sequence spaces,  in which the canonical basis is   subsymmetric but not symmetric. In this note we further the study of the greedy behavior of subsymmetric bases and investigate  Garling sequence spaces from the point of view of the greedy algorithm. To be precise 
in Section~\ref{Renorming} we prove  that  the canonical basis of Garling sequence spaces is  $1$-GUR. In  Section~\ref{Superreflexivity} we use the properties of the democracy functions of these spaces to give a necessary condition for them to be super-reflexive. In addition, we prove that Garling sequence spaces are never uniformly convex.

It is  worth pointing out that  investigating greedy renormings of non-sub\-symmetric bases is also of interest. Indeed,  the starting problem of  this theory,  posed in \cite{AW2006} and as of today still unsolved, is to determine if the  Haar system  in $L_p[0,1]$, $1<p<\infty$,
is a $1$-GUR basis. Recall that the Haar system in $L_p[0,1]$  is greedy \cite{Temlyakov1998} but it is not subsymmetric \cite{KadetsPel1962}.
The most significant advances  in the study of greedy renormings  of non-subsymmetric bases  were also achieved in \cite{DOSZ2011}. Here the authors found  examples of   non-subsymmetric  greedy bases which are not $1$-GUR
 (like the Haar basis in the dyadic Hardy space $H_1$ and the canonical basis of the Tsirelson space),  and of a  non-subsymmetric greedy  basis    which is  $1$-GUR (namely, the canonical basis of the space $\ell_2\oplus\ell_{2,1}$).
 
Throughout this article we use standard facts and notation from Banach spaces and approximation theory. We refer the reader to e.g. \cites{AlbiacKalton2016, LinTza1977, LinTza1979} for the necessary background. Next  we single out the notation that it is more heavily employed.
We will  denote by $\FF$ the real or complex field. We denote by
 $(\ee_k)_{k=1}^\infty$  the canonical basis of $\FF^\NN$, i.e.,  $\ee_k=(\delta_{k,n})_{n=1}^\infty$, were
$\delta_{k,n}=1$ if $n=k$ and $\delta_{k,n}=0$ otherwise.
The domain of a function $f$ will be denoted by $\dom(f)$, while $\range(f)$ denotes its range.
Given families of positive real numbers $(\alpha_i)_{i\in I}$ and $(\beta_i)_{i\in I}$, the symbol $\alpha_i\lesssim \beta_i$ for $i\in I$ means that $\sup_{i\in I}\alpha_i/\beta_i <\infty$, while $\alpha_i\approx \beta_i$ for $i\in I$ means that $\alpha_i\lesssim \beta_i$ and $\beta_i\lesssim \alpha_i$ for $i\in I$.


\section{Superreflexivity in Garling sequence spaces}\label{Superreflexivity}
 \noindent  Let us consider the  set of weights  
 \[
 \WW:=\left\{(w_n)_{n=1}^\infty\in c_0\setminus\ell_1:1=w_1\geq w_2>\cdots w_n \ge w_{n+1} \ge\cdots>0\right\}.\]
 Given $1\le p<\infty$ and $\ww=(w_n)_{n=1}^\infty\in\WW$ the \textit{Garling sequence space} $g(\ww,p)$ is defined as the Banach space consisting of all scalar sequences $f=(a_n)_{n=1}^\infty$ such that
\[
 \Vert  f  \Vert_{g(\ww,p)} = \sup_{\phi\in\OO_\infty } \left( \sum_{n=1}^\infty |a_{\phi(n)}|^p w_n \right)^{1/p},
 \] 
 where  $\OO_\infty$ denotes the set of all increasing functions from $\NN$ to $\NN$.
If $\ww$ and $p$ are clear from context, the norm of the space will be shortened to $\Vert\cdot\Vert_{g}$. The isomorphic structure of these  Banach spaces, which generalize an example of Garling from \cite{Garling1968}, has been recently  studied in \cite{Wallis2017}. 

Theorem~\ref{PreviousResults} below gathers a few properties of Garling sequence spaces that are of interest for the purposes of this paper.

Recall that given a basis  $\BB=(\xx_n)_{n=1}^\infty$ for a Banach space $X$, the \textit{lower democracy funtion} $(\varphi_l[\BB, X](m))_{m=1}^\infty$   and the \textit{upper democracy funtion} $(\varphi_u[\BB, X](m))_{m=1}^\infty$   of $\BB$ are defined, respectively,  by 
 \[
\varphi_l[\BB,X](m) =\inf_{|A|\ge m} \left\Vert \sum_{n\in A} \xx_n  \right\Vert,\]
and 
\[\varphi_u[\BB,X](m) =\sup_{|A|\le m} \left\Vert \sum_{n\in A} \xx_n  \right\Vert,
\]
and that  $\mathcal B$ is $\Delta$-democratic if and only if $\varphi_u[\BB, X](m)\le\Delta \varphi_l[\BB,X](m)$ for all $m\in\NN$. 

Recall also that a weight   $(w_n)_{n=1}^\infty$ is said to be  \textit{regular} if there is a constant
 $C\ge 1$ such that 
 \[
 \frac{1}{m} \sum_{n=1}^m w_n\le Cw_m, \quad m\in\NN.
 \]

\begin{theorem}[see \cite{Wallis2017}]\label{PreviousResults}  Let  $1\le p<\infty$ and $\ww=(w_n)_{n=1}^\infty\in \WW$. Then:
\begin{itemize}

\item[(i)]  The canonical basis $\EE=(\ee_n)_{n=1}^\infty$  is a $1$-subsymmetric 
basis of $g(\ww,p)$.

\item[(ii)]  If both $\ww$ and $(1/(nw_n))_{n=1}^\infty$ are regular weights then $\EE$ is not symmetric in $g(\ww,p)$.

\item[(iii)] $\varphi_l[\EE, g(\ww,p)](m) = \varphi_u[\EE, g(\ww,p)](m) = (\sum_{n=1}^m w_n)^{1/p}$  for all $m\in\NN$.

\item[(iv)] $g(\ww,p)$ is reflexive if and only if $p>1$.

\item[(v)] Any subsymmetric basis of $g(\ww,p)$ is equivalent to its canonical basis.

\item[(vi)] For every  $\varepsilon>0$ there is a  sublattice of $g(\ww,p)$ that is $(1+\varepsilon)$-lattice isomorphic to $\ell_p$ and 
$(1+\varepsilon)$-lattice complemented in $g(\ww,p)$.


\end{itemize}
\end{theorem}

Let us get started by  using the democracy functions to obtain some embedding results.
\begin{proposition}\label{Embedding} Let  $1\le p<\infty$. Let  $\vv=(v_n)_{n=1}^\infty$,
$\ww=(w_n)_{n=1}^\infty\in\WW$ with $\ww$ regular.
Then  $g_p(\ww)\subseteq g_p(\vv)$ if and only if 
$v_n\lesssim w_n$ for $n\in\NN$.
\end{proposition}

  \begin{proof}  If $g_p(\ww)\subseteq g_p(\vv)$ the embedding is continuous and so 
  \[
  \varphi_u[\EE, g(\vv,p)](m) \lesssim   \varphi_u[\EE, g(\ww,p)](m), \quad m\in\NN.
  \]
Appealing to Theorem~\ref{PreviousResults}~(iii) we get
  \[
  v_m 
  \le \frac{1}{m}\sum_{n=1}^m  v_n
 \lesssim 
  \frac{1}{m}\sum_{n=1}^m  w_n  
   \lesssim w_m, \quad m\in\NN.
 \]
The converse is obvious. 
  \end{proof}
 
  

  \begin{corollary} Let  $1\le p<\infty$. Let  $\vv=(v_n)_{n=1}^\infty$,
$\ww=(w_n)_{n=1}^\infty\in\WW$.
  \begin{itemize}
 \item[(i)]   $g_p(\vv) \approx g_p(\ww)$ if and only if  $g_p(\vv) =  g_p(\ww)$.

\item[(ii)]  Assume that both $\vv$ and $\ww$ are regular and that $g_p(\vv) =  g_p(\ww)$. Then
   $v_n \approx w_n$ for $n\in\NN$.
  \end{itemize}
  \end{corollary}
  \begin{proof} 
  (i) is a consequence of Theorem~\ref{PreviousResults}~(v), and (ii) is straightforward from 
  Proposition~\ref{Embedding}.
  \end{proof}


\begin{proposition}\label{NoUniformConvex}The space $g(\ww,p)$ fails to be uniformly convex for any $1\leq p<\infty$ and any $\ww\in\WW$.\end{proposition}

\begin{proof}For $j\in\NN$ put
\[\alpha_j=\left(\frac{1-w_{j+1}}{\sum_{n=1}^jw_n}\right)^{1/p},\]
and consider the vectors
\[u^{(j)}=(\underbrace{\alpha_j,\cdots,\alpha_j}_{j\text{ times}},1,0,0,0,\cdots) \]
and
\[v^{(j)}=(\underbrace{\alpha_j,\cdots,\alpha_j}_{j\text{ times}},0,1,0,0,0,\cdots).\]
Observe that
\begin{align*}\frac{1}{2}\norm{u^{(j)}+v^{(j)}}_{g}
&\geq\frac{1}{2}\left[\sum_{n=1}^{j}(2\alpha_{j})^pw_n+w_{j+1}+w_{j+2}\right]^{1/p}\\
&=\left[\alpha_j^p\sum_{n=1}^jw_n+\frac{1}{2^p}(w_{j+1}+w_{j+2})\right]^{1/p}\\
&=\left[1-w_{j+1}+\frac{1}{2^p}(w_{j+1}+w_{j+2})\right]^{1/p}\\
&:=N_j.
\end{align*}
Since $\lim_j N_j=1$, to show that $g(\ww,p)$ fails to be uniformly convex, it suffices to find an increasing sequence of integers $(j_k)_{k=1}^\infty$ such that $\norm{u^{(j_k)}}_{g}=\Vert v^{(j_k)}\Vert_{g}=1$ and $\Vert u^{(j_k)}-v^{(j_k)}\Vert_{g}>1$ for all $k\in\NN$.

Due to $(w_j)_{j=1}^\infty\in c_0\setminus\ell_1$ we have
\[\lim_{j\to\infty}\alpha_j=0.\]
Hence, we could find a subsequence $(\alpha_{j_k})_{k=1}^\infty$ such that
\[\alpha_{j_k}\leq\min_{i\leq j_k}\alpha_i,\;\;\;k\in\NN.\]

Now, fix any $k\in\NN$.  By definition of $g(w,p)$ and due to $w_1=1$, either $\norm{u^{(j_k)}}_{g}=1$, or else we could find $i\in\{1,\cdots,j_k\}$ with
\[1\leq\norm{u^{(j_k)}}_{g}^p=\sum_{n=1}^iw_n\alpha_{j_k}^p+w_{i+1}\leq\sum_{n=1}^iw_n\alpha_i^p+w_{i+1}=1\]
so that $\norm{u^{(j_k)}}_{g}=\norm{v^{(j_k)}}_{g}=1$ anyway.  Observing that
\[\norm{u^{(j_k)}-v^{(j_k)}}_{g}=(w_1+w_2)^{1/p}>1\]
finishes the proof.\end{proof}

Enflo proved in \cite{Enflo1973} that a Banach space is superreflexive if and only if it is uniformly convex under a suitable renorming.
Having shown that $g(\ww,p)$ is never uniformly convex, and in light of the above identification between superreflexivity and uniform convexifiability, the next natural question to ask is: Given $1 < p < \infty$, can we ever choose $\ww\in\WW$ so that $g(\ww,p)$ is superreflexive? 

We tackle this issue by using well-known properties of the democracy functions of bases in Banach spaces. Following \cite{DKKT2003} we say 
that a sequence $(s_n)_{n=1}^\infty$  of positive numbers has the \textit{lower regularity property} (LRP for short) if  there is  an integer $r\ge 2$ with
\[
s_{rn}\ge 2 s_n, \quad n\in\NN.
\]
Our next Proposition establishes the close relation between a weight $(w_n)_{n=1}^\infty$ being regular and its \textit{primitive weight} $(s_n)_{n=1}^\infty$ given by
$s_n=\sum_{k=1}^n w_k$ having the LRP. Recall that $(w_n)_{n=1}^\infty$ is 
\textit{essentially decreasing} if there is a constant $C\ge 1$ with   
 $w_k\le Cw_n$ for $k\ge n$.

\begin{proposition}\label{result:3} Let $(s_n)_{n=1}^\infty$ be the primitive weight of an essentially decreasing weight $(w_n)_{n=1}^\infty$. The following are equivalent.
\begin{itemize}

\item[(a)] There is $C>1$ such that $s_{2n}\ge C s_n$ for all $n\in\NN$.

\item[(b)] For every $C>1$ there is $r\in\NN$ with $s_{rn}\ge C s_n$ for all $n\in\NN$.

\item[(c)] $(s_n)_{n=1}^\infty$ has the LRP.

\item[(d)] There is $C>1$ and $r\in\NN$  with  $s_{rn}\ge C s_n$ for all $n\in\NN$.

\item[(e)]  There exists $a>0$ such that 
$(n^{-a} s_n)_{n=1}^\infty$ is essentially increasing.

\item[(f)]  
$(n^{-1} s_n)_{n=1}^\infty$ is a regular weight.

\item[(g)]  $(w_n)_{n=1}^\infty$ is a regular weight.

\end{itemize}

\end{proposition}

\begin{proof}Taking into account \cite{Altshuler1975}*{Theorem 1} and
\cite{AA2016}*{Lemma 2.12}, we must only prove (a) $\Rightarrow$ (g). Assume that $s_{rn}\ge C  s_n$ for some $C>1$, some $r\ge 2$ and all $n\in\NN$.  Let $D=\sup_{k\le n} w_n/w_k$.
We have
\[
\frac{nw_n}{s_n}\ge  \frac{1}{D(r-1)} \frac{s_{rn}-s_n}{s_n}\ge \frac{C-1}{D(r-1)}
\]
for all $n\in\NN$.
\end{proof}

\begin{lemma}\label{convexity}
Let $1\le p<\infty$ and $\ww\in\WW$.  Then $g(\ww,p)$ is $p$-convex and it is not $q$-convex for any
$q>p$.
\end{lemma}

\begin{proof}By Theorem~\ref{PreviousResults}~(vi), the space $g(\ww,p)$ contains $\ell_p$ as a sublattice hence
it is not $q$-convex  for any $q>p$.   Showing that $g(\ww,p)$ is $p$-convex is straightforward.  \end{proof}

\begin{proposition}\label{SECotype}Let $1< p<\infty$ and $\ww\in\WW$.  The following are equivalent.
\begin{itemize}
\item[(a)] $g(\ww,p)$ is superreflexive.
\item[(b)] $g(\ww,p)$ has non-trivial cotype. 
\item[(c)] $g(\ww,1)$ has non-trivial cotype. 
\end{itemize}
\end{proposition}

\begin{proof} (b) $\Rightarrow$ (c)
Assume that $g(\ww,p)$ has cotype $q$ for some $q<\infty$. Then,  $g(\ww,p)$ satisfies  satisfies a lower $q$-estimate. Since 
\[ \Vert f\Vert_{g(\ww,1)}= \Vert |f|^{1/p} \Vert_{g(\ww,p)}^{p},\]
it follows that $g(\ww,1)$ satisfies a lower $q/p$-estimate. By \cite{LinTza1979}*{Proposition 1.f.3 and Theorem 1.f.7}, $g(\ww,1)$ has cotype $r$ whenever $r\ge 2$ and $r>q/p$.

(c) $\Rightarrow$ (a)  Assume that $g(\ww,1)$ has cotype $r<\infty$. Arguing as before, we claim that $g(\ww,p)$ satisfies a lower  $pr$-estimate. Taking into account Lemma~\ref{convexity}, we infer from  \cite{LinTza1979}*{Theorem 1.f.10} that $g(\ww,p)$ is superreflexive.

(a) $\Rightarrow$ (b) is a well known consequence of  \cite{MP1976}*{Theorem 1.1}.
\end{proof}

The key ingredient in the proof of the next theorem is the link between the (Rademacher)
 type/cotype  of a space and the regularity properties of  the democracy functions of its almost greedy bases (see \cite{DKKT2003}).
\begin{theorem}\label{SRvsRegularity} Let $1< p<\infty$ and $\ww\in\WW$ be such that $g(\ww,p)$ is superreflexive. Then $\ww$ is a regular weight.
\end{theorem}
\begin{proof}
  The space  $g(\ww,1)$ has finite cotype by Proposition~\ref{SECotype}. Combining
\cite{DKKT2003}*{Proposition 4.1} and Theorem~\ref{PreviousResults}~(c) yields that
$(\sum_{n=1}^m w_n)_{m=1}^\infty$ has the LRP. Then, by Proposition~\ref{result:3}, $\ww$ is a regular weight.
\end{proof}

\begin{corollary}\label{loofiniterepresentability} Let $1\le p<\infty$ and $\ww\in\WW$ be non-regular. Then $\ell_\infty$ is finitely representable in 
$g(\ww,p)$. 
\end{corollary}
\begin{proof}By Corollary~\ref{SRvsRegularity}, $g(\ww,2)$ is not superreflexive. Then, by Proposition~\ref{SECotype}, $g(\ww,p)$ has trivial cotype.  The proof is over by appealing to \cite{MP1976}*{Theorem 1.1}.
\end{proof}

\begin{remark}
 Corollary~\ref{loofiniterepresentability} could alternatively be shown by following the steps of the proof from
\cite{Altshuler1975}  that $d(\ww,p)$ is not superreflexive if $\ww$ fails to be regular.  Altshuler's  method leads to the following result: for each $p>1$, each non-regular weight $\ww$, each $\varepsilon>0$, and each $k\in\NN$ there is a constant-coefficient finite block basic sequence  of the canonical basis of $g(\ww,p)$ that is $(1+\varepsilon)$-equivalent to the canonical basis of $\ell_\infty^k$.
We would  also like to point out that the fact that $d(\ww,p)$ is superreflexive only if $\ww$ is regular can  be obtained using intrinsic ideas from this manuscript.
\end{remark}

\section{Greedy renormimgs of Garling sequence spaces}\label{Renorming}

\noindent 
Given a basis $(\xx_n)_{n=1}^\infty$  for $X$ and  $f$, $g$ in $X$ we say that $g$ is a \textit{greedy permutation} of $f$  if we can write
\begin{equation}\label{eq:greedypermutation}
f=h+t \sum_{n\in A} \varepsilon_{n} \xx_{n}\quad\text{and}\;\; g=h+ t \sum_{n\in B}\theta_{n} \xx_{n}
\end{equation}
for   some $h\in X$,
some sets of integers $A$ and $B$ of the same finite cardinality with $\supp(h)\cap (A\cup B)=\emptyset$,  some signs $(\varepsilon_{n})_{n\in A}$ and $(\theta_{n})_{n\in B}$, and some scalar $t$ such that
 $\sup_n |\xx_n^*(h)|\le t$. If, in addition, $A \cap B=\emptyset$, we say that $g$ is a 
 \textit{disjoint greedy permutation} of $f$.
 In other words,  $g$ is a  disjoint greedy permutation of $f$  if $g$ is obtained from $f$ by moving  those  terms of $f$ (or some of them) whose coefficients are  maximum in absolute value to  gaps in the support of $f$. We are also allowed to change the sign of (some of) the terms we move. Then, the basis
$(\xx_n)_{n=1}^\infty$ is said to  satisfy \textit{Property (A)} if $\Vert f\Vert=\Vert g\Vert$ whenever $g$ is a  disjoint greedy permutation of $f$.
Actually,   $(\xx_n)_{n=1}^\infty$ has Property (A) if and only if whenever $g$ is a greedy permutation of $f$ then $\Vert g \Vert =\Vert f\Vert$ (which is the way Property (A) was originally defined in \cite{AW2006}).
Property (A)  is  stronger than democracy. 
Albiac and Wojtaszczyk \cite{AW2006} proved that a basis is $1$-greedy if and only if is $1$-suppression unconditional and has Property (A).

As an immediate consequence of Theorem~\ref{PreviousResults}~(i) we obtain that the canonical basis of
$g(\ww,p)$ is $2$-greedy. However, it   is never   $1$-greedy as we see next.

\begin{lemma} The canonical basis of $g(\ww,p)$, $1\le p<\infty$ and $\ww\in\WW$, is not  $1$-greedy. 
\end{lemma}

\begin{proof}Choose $k\in\NN$ and $v\in(0,\infty)$ with $w_n=1$ for $1\le n\le k$ and $w_{k+1}=v<1$.
Pick  $t> 1$ and put $f=t\ee_1+\sum_{n=2}^{k+1} \ee_k$ and  $g=t\ee_{k+2}+\sum_{n=2}^{k+1} \ee_k$. Consider for each $j\in\NN\cup\{0\}$ the translation map $\phi_j\in\OO$ given by $\phi_j(n)=n+j$. Let
\[
x:=\Vert f \Vert_g^p
=\max_{j\in\{0,k\}} \Vert f\circ\phi_j \Vert_{p,\ww}^p
=\max\{ t+k-1+v,tv\},\]
and
\[
y:=\Vert g \Vert_g^p
=\max_{j\in\{1,2\}} \Vert g\circ\phi_j \Vert_{p,\ww}^p
=\max\{ k+tv,t+k-1\}.
\]
Notice that $g$ is a greedy rearrangement of $f$. Hence, assuming that $(\ee_n)_{n=1}^\infty$ is $1$-greedy,
yields $x=y$.
We infer that $x= t+k-1+v$ and $y= k+tv$. Then
we reach the absurdity $v=1$.
\end{proof}

In order 
 to give more relevance to Theorem~\ref{MainTheorem}, it would be convenient  to recall that under a natural condition on the weight $\ww$ the canonical basis is not a symmetric basic sequence of $g(\ww,p)$
 (see Theorem~\ref{PreviousResults}~(ii)).

\begin{theorem}\label{MainTheorem}Let $1\le p<\infty$ and $\ww\in\WW$ be a regular weight. Then there is a renorming of $g(\ww,p)$ with respect to which the canonical basis is   $1$-greedy and  $1$-subsymmetric.
\end{theorem}
Before proving Theorem~\ref{MainTheorem} we shall introduce some additional notation. 
Suppose $1\le p<\infty$, and let  $\ww=(w_i)_{i\in I}$ be a family of positive scalars.
Given a family of scalars $f=(a_i)_{i\in A}$, where $A\subseteq I$, we put
\[
\Vert f \Vert_{p,\ww}=\left(\sum_{i\in A} |a_i|^p w_i\right)^{1/p}.
\]
Given $r\in\NN$, denote by $\OO_r$ the set of all increasing functions from the integer interval $[0,r]\cap \ZZ$ into $\NN$. Put
$\OO_f=\cup_{r=1}^\infty \OO_r$ and $\OO=\OO_f\cup\OO_\infty$. Note that for all $f\in\FF^\NN$, \[
\Vert f \Vert_g=\sup_{\phi  \in \OO_f} \Vert f\circ \phi \Vert_{p,\ww}=\sup_{\phi  \in \OO_\infty} \Vert f\circ \phi \Vert_{p,\ww}
=\sup_{\phi  \in \OO} \Vert f\circ \phi \Vert_{p,\ww}.
\]
where  $f\circ \phi =  (a_{\phi(n)})_{n\in \dom(\phi)}$.
Let  $\HH$  be the set of all increasing functions from a subset of $\NN$ into $\NN$.
 Given $\beta\in\HH$
 consider the linear operator $U_{\beta}\colon\FF^\NN\to \FF^\NN$ defined
 by  $U_{\beta}(f)=(b_n)_{n=1}^\infty$, where, if $f= (a_n)_{n=1}^\infty$,
\begin{equation*}
 b_n=\begin{cases} 
 a_{\beta(n)}& \text{ if } n\in \dom(\beta), \\
 0 & \text{otherwise.}
 \end{cases}
\end{equation*}
Note that  if the canonical basis $(\ee_n)_{n=1}^\infty$ of a sequence space $X$ is $1$-unconditional and verifies  $\sup_{\beta\in\HH} \Vert U_\beta \colon  X \to X\Vert\le C$ 
then  $(\ee_n)_{n=1}^\infty$  is a  $C$-subsymmetric basic sequence in $X$.
\begin{proof}[Proof of Theorem~\ref{MainTheorem}] Let $\ww=(w_n)_{n=1}^\infty$  and put 
\[
D=\sup_m \frac{1}{mw_m}\sum_{n=1}^m   w_n.
\]
For $m\in\NN$ define 
 \[t_m=\frac{2}{m}\sum_{n=1}^m w_n.
 \]
 We have that $(t_m)_{m=1}^\infty$ is non-increasing, that $t_m\le 2 D  w_r$ for 
$r\le m$, and that
$
(m+1)t_{m+1}-m t_m=2 w_{m+1}.
$

Given $m\in\NN\cup\{0\}$, let us denote by $\FFF_m$ the set of pairs $(A,\alpha)$, where $A\subseteq \NN$ and $\alpha\in\HH$
 verify $|A|=m$, $\dom(\alpha)\subseteq[m+1,\infty)$, and  $\range(\alpha)\cap A=\emptyset$.

Consider also   the set $\FFF_m'$ of triads $(\rho,\phi,\psi)$, where 
$\rho\in\OO_m$,  $\phi,\psi\in\OO_r$ for some $r\in\NN\cup\{\infty\}$,
$\range(\psi)\subseteq  [m+1,\infty)$ and $\range(\rho)\cap \range(\phi)=\emptyset$.
Note that the mapping $(\rho,\phi,\psi)\mapsto(A,\alpha)$ where $A$ and $\alpha$ are determined by
 \begin{equation}\label{eq:7}
A=\range(\rho), \quad \alpha(\psi(n))=\phi(n) \text{ for all } n\in\dom(\phi),
\end{equation}
 is a bijection from $\FFF_m'$ onto $\FFF_m$.
 
Given $(A,\alpha)\in\FFF_m$  and  $f\in\FF^\NN$, we define
\[
\Vert f \Vert_{A,\alpha}=
\left( t_m \Vert f|_A \Vert_p^p + \Vert f\circ \alpha \Vert^p_{p,\ww}  \right)^{1/p}.
\]
Let  $(\rho,\phi,\psi)$ be the element in $\FFF'_m$ that corresponds to $(A,\alpha)$ by the relation \eqref{eq:7}. We have
\begin{align*}
\Vert f \Vert_{A,\alpha}
&= \left(  t_m \Vert f\circ \rho  \Vert_p^p + \Vert f\circ \phi \Vert^p_{p,\ww\circ \psi} \right)^{1/p}\\
&\le \left( 2 D \Vert f\circ \rho \Vert^p_{p,\ww} +\Vert f\circ \phi \Vert^p_{p,\ww}\right)^{1/p}\\
& \le\left( 2 D+ 1\right)^{1/p} \Vert f \Vert_g
\end{align*}
 Put $\FFF=\cup_{m=0}^\infty \FFF_m$ and define
 \[
 \Vert f\Vert =\sup_{(A,\alpha)\in\FFF} \Vert f \Vert_{A,\alpha}, \quad f\in\FF^\NN.
 \]
We have 
\[\Vert f\Vert \le (2D+1)^{1/p}  \Vert f \Vert_g\]
 and
\[
 \Vert f\Vert 
 \ge \sup_{(A,\alpha)\in\FFF_0} \Vert f \Vert_{A,\alpha}
= \sup_{\alpha\in\HH} \Vert f \circ\alpha \Vert_{p,\ww} 
 \ge \sup_{\alpha\in\OO}  \Vert f \circ\alpha \Vert_{p,\ww}
=\Vert f \Vert_g.
 \]
 Hence $(g(\ww,p),\Vert \cdot\Vert)$ is a renorming of $g(\ww,p)$.
 
 Next we go on to substantiate the following Claim:
 
\noindent\textbf{Claim.}
Let $f=(a_n)_{n=1}^\infty \in \FF^\NN$ and $k\in \NN$  such that $|a_k|\ge |a_n|$ for every $n\in\NN$. Then
\[
\Vert f\Vert = \sup\{  
\Vert f \Vert_{A,\alpha} \colon (A,\alpha)\in\FFF,  k\in A,  A \cup\range(\alpha)\subseteq\supp(f)
\}.
\]

\medskip\noindent
 Assume, without loss of generality that $|a_k|=1$.
Pick $(A,\alpha)\in\FFF_m$ for some $m\in\NN$. Let $B=A\cap\supp f$ and $\beta$ be the restriction of 
$\alpha$ to $\alpha^{-1}(\supp(f))$. We have
$(B,\beta)\in\FFF_r$ for some $r\le m$,  that
\[
x:=\Vert f|_A \Vert_p =\Vert f|_B \Vert_p,
\]
and  that 
\[\Vert f\circ\alpha \Vert_{p,\ww} =\Vert f\circ \beta \Vert_{p,\ww}.
\] Hence
$\Vert f \Vert_{A,\alpha}\le \Vert f \Vert_{B,\beta}$.

If $k\in B$ we are done.
 Assume that $k\notin B$. Let
$
 E=B\cup\{k\}
 $
 and 
$ \gamma$ be the restriction of $\beta$ to $\beta^{-1}(\NN\setminus\{k\})\cap[r+2,\infty)$.
Notice that $(E,\gamma)\in\FFF_{r+1}$.
If $r+1\in \dom(\beta)$ put 
\[
u=w_{r+1} \text{ and }y=|a_{\beta(r+1)}|,
\] and, otherwise, put $y=u=0$.
 If there is a (unique) $j\ge r+2$ with $\beta(j)=k$, put 
 \[
 v=w_j \text{ and } z=|a_{\beta(j)}|,
 \] and, otherwise, put $z=v=0$.
 Taking into account that $x^p \le p$ and that $y^p,z^p\le 1$, and that $t_r\le t_{r+1}$,
\begin{align*}
 \Vert f \Vert_{B,\beta}^p - \Vert f \Vert_{E,\gamma}^p
& = t_r \,  x^p - t_{r+1} (1+x^p) +y^p u+ z^p v \\
&\le t_r\, p-  t_{r+1} (1+ p) +u+v \\ 
&=-2w_{r+1}+ u+v \\
&\le-2w_{r+1}+ w_{r+1}+w_{r+2} \\ 
&\le  0,
\end{align*}
as desired.

Now we are ready to prove that $(\ee_n)_{n=1}^\infty$ is $1$-greedy with respect to the norm $\Vert \cdot\Vert$.
Since it is $1$-unconditional, we must only show that it has  Property (A). To that end if suffices to see that
\begin{equation}\label{eq:23}
\Vert \ee_k+f\Vert \le \Vert \varepsilon \ee_j+f\Vert
\end{equation}
for every 
 $f\in \FF^\NN$ with  $\Vert f \Vert_\infty\le 1$, every sign $\varepsilon$,  and every $j,k\notin\supp (f)$ 
 with $j\not=k$. 
 
 In order to compute  $\Vert \ee_k+f\Vert$, taking into account the Claim, we can and we do restrict our attention  to  $(A,\alpha)\in\FFF$ with $k\in A$ and $A\cup\range(\alpha)\subseteq \{k\}\cup\supp(f)$.
In particular,  we have $ j\not \in A   \cup\range(\alpha)$. Choose $B=(A\cup\{j\})\setminus\{ k\}$.
 We have $(B,\alpha)\in\FFF$ and 
 \[\Vert (\ee_k + f )|_A \Vert_p = \Vert (\varepsilon\ee_j+ f)|_B \Vert_p.\]
Hence,
\[
\Vert \ee_k+f \Vert_{A,\alpha}
=\Vert \varepsilon\ee_j+f \Vert_{B,\alpha}
\le \Vert \varepsilon\ee_j +f\Vert.
 \]
 We  obtain \eqref{eq:23} by taking the supremum on $(A,\alpha)$.
 
 Let us prove that the canonical basis is $1$-subsymmetric with respect to the norm
 $\Vert \cdot \Vert$.
Let  $\beta\in\HH$,  $f\in\FF^\NN$ and $(A,\alpha)\in\FFF$.
Since  $|\beta(A)|\le |A|$, we have
 $(\beta(A), \beta\circ\alpha)\in\FFF$. Moreover
\[
\Vert U_{\beta}(f)|_A \Vert_p =\Vert f|_{\beta(A)} \Vert_p,
\]
and
\[
\Vert U_{\beta}(f) \circ\alpha \Vert_{p,\ww} =\Vert f\circ \beta\circ\alpha \Vert_{p,\ww},
\]
so that
$
 \Vert  U_{\beta}(f)\Vert_{A,\alpha}  \le \Vert f\Vert.
 $
 Consequently,   $\Vert U_{\beta}(f)\Vert \le \Vert f\Vert$. 
 \end{proof}

 \begin{problem}  Does every Banach space with  a subsymmetric basis admit a $1$-greedy renorming?
 \end{problem}


\end{document}